\documentclass[12pt]{article}

\usepackage{tikz}
\usepackage{subfigure}
\usepackage[english]{babel}
\usepackage{amsfonts,amssymb,amsmath,latexsym,amsthm}
\usepackage{multirow}

\newtheorem{thm}{Theorem}[section]
\newtheorem{cor}[thm]{Corollary}
\newtheorem{lem}[thm]{Lemma}
\newtheorem{prop}[thm]{Proposition}

\newtheorem{claim}{Claim}

\newtheorem{obs}[thm]{Observation}

\begin{document}

\title{A Dirac-type theorem for uniform hypergraphs 
\thanks{The work was supported by National Nature Science Foundation of China (No. 11671376) and Anhui Initiative in Quantum Information Technologies (AHY150200).}
}
\author{Yue Ma$^a$, \quad Xinmin Hou$^b$,\quad Jun Gao$^c$\\
\small $^{a,b,c}$ Key Laboratory of Wu Wen-Tsun Mathematics\\
\small School of Mathematical Sciences\\
\small University of Science and Technology of China\\
\small Hefei, Anhui 230026, China.
}

\date{}

\maketitle

\begin{abstract}
Dirac (1952) proved that every connected graph of order $n>2k+1$ with minimum degree more than $k$ contains a path of length at least $2k+1$.
Erd\H{o}s and Gallai (1959) showed that every $n$-vertex graph $G$ with average degree more than $k-1$ contains a path of length $k$.
The hypergraph extension of the Erd\H{o}s-Gallai Theorem  have been  given by Gy\H{o}ri, Katona, Lemons~(2016) and Davoodi et al.~(2018).
F\"uredi, Kostochka, and Luo (2019) gave a connected version of the Erd\H{o}s-Gallai Theorem for hypergraphs. In this paper, we give a hypergraph extension of the  Dirac's Theorem:
Given positive integers $n,k$ and $r$, let $H$ be a connected $n$-vertex $r$-graph with no Berge path of length $2k+1$. We show that
(1) If $k> r\ge 4$ and $n>2k+1$, then $\delta_1(H)\le\binom{k}{r-1}$. Furthermore, the equality holds if and only if  $S'_r(n,k)\subseteq H\subseteq S_r(n,k)$ or $H\cong S(sK_{k+1}^{(r)},1)$;
(2) If $k\ge r\ge 2$ and $n>2k(r-1)$, then $\delta_1(H)\le \binom{k}{r-1}$.
The result is also a Dirac-type version of the result of F\"uredi, Kostochka, and Luo.
As an application of (1), we give a better lower bound of the minimum degree than the ones  in the Dirac-type results for Berge Hamiltonian cycle given by  Bermond et al.~(1976) and Clemens et al. (2016), respectively.
\end{abstract}

\section{Introduction}
An $r$-uniform hypergraph, or $r$-graph, is a pair $H=(V,E)$, where $V$ is a set of elements called vertices, and $E$ is a collection of subsets of $V$ with uniform size $r$ called edges.
In this article, all $r$-graphs $H$ considered are simple, i.e. $H$ contains no multiple edges.
We call $|V|$ {\it the order} of $H$ and $|E|$ {\it the size} of $H$, also denoted by $|H|$ or $e(H)$.  We write graph for $2$-graph for short.
Given $S\subseteq V(H)$, the {\it degree} of $S$, denote by $d_H(S)$, is the number of edges of $H$ containing $S$. The minimum $s$-degree $\delta_s(H)$ of $H$ is the minimum of $d_H(S)$ over all $S\subseteq V(H)$ of size $s$.
We call $\delta_1(H)$ the \emph{minimum degree} of $H$, that is $\delta_1(H)=\min\{d_H(v) : v\in V(H)\}$. Let $N_H(S)=\{ T : S\cup T\in E(H)\}$.
 Given two integers $a,b$ with $a<b$, write $[a,b]$ for the set $\{a, a+1, \ldots, b\}$.

The following two theorems, due to Dirac~\cite{Dirac52} and Erd\H{o}s and Gallai~\cite{EG59}, are well-known  in  graph theory.

\begin{thm}[Dirac, 1952]\label{THM: Dirac}
Let $G$ be a connected graph on $n$ vertices with minimum degree $\delta_1(G)>k$. If $n>2k+1$, then $G$ contains a path of length at least $2k+1$.
\end{thm}

\begin{thm}[Erd\H{o}s-Gallai Theorem, 1959]\label{THM: EGT}
Let $G$ be a graph on $n$ vertices with $e(G)>\frac{(k-1)n}2$ (or $e(G)>\frac{(k-1)(n-1)}2$). Then  $G$ contains a path of length $k$ ( or a cycle of length at least $k$).
\end{thm}

The type of problems that relate the ($d$-)minimum degree (resp. the number of edges) in (hyper)graphs to the structure of the (hyper)graphs are often referred to as Dirac-type (resp. Tur\'an-type) problems. These types of problems for hypergraphs have received much attention in recent years, see~\cite{Kee11,RR10,Zh16} for the surveys.

A {\it Berge path} $P$ of length $t$ in a hypergraph is a collection of $t+1$ distinct vertices $\{v_0,v_1,...,v_{t}\}$ and $t$ distinct edges $\{e_1,e_2,...,e_t\}$  such that $\{v_{i-1},v_i\} \subseteq e_i$ for $1\le i\le t$.
A {\it Berge cycle} $C$ of length $t$ in a hypergraph is a collection of $t$ distinct edges $\{e_1,e_2,...,e_t\}$ and $t$ distinct vertices $\{v_1,...,v_{t}\}$ such that $\{v_{i-1},v_i\} \subseteq e_i$ with indices taken modulo $t$. We call $\{v_0,...,v_t\}$ (resp. $\{v_1,...,v_t\}$) the key vertices of $P$ (resp. $C$), denoted by $K(P)$ (resp. $K(C)$), and $P$ is called a Berge path connecting $v_0$ and $v_t$.
An $r$-graph $H$ is called connected if for any two vertices $u,v\in V(H)$, there exists a Berge path $P$ connecting $u$ and $v$.

The hypergraph extension of the Erd\H{o}s-Gallai Theorem (Theorem~\ref{THM: EGT}) have been  solved completely in two recent papers by Gy\H{o}ri, Katona, Lemons~\cite{GKL16} and Davoodi et al.~\cite{DGMT18}.
\begin{thm}[Theorem 1.3 in~\cite{GKL16} and Theorem 3 in~\cite{DGMT18}]\label{THM: Gyori&Davoodi}
Let $H$ be an $n$-vertex $r$-graph with no Berge path of length $k$. If $r>k>3$, then $e(H)\le\frac{(k-1)n}{r+1}$.
If $k\ge r+1>3$ then $e(H)\le \frac nk{k\choose r}$. Furthermore, these bounds are sharp for each $k$ and $r$ for infinitely many $n$.
\end{thm}

We should mention that the extremal hypergraphs are disconnected in Theorem~\ref{THM: Gyori&Davoodi}. The connected version of the Erd\H{os}-Gallai Theorem have also received much attention for graphs (see~\cite{BGJS08,Kop77,Luo18}) and hypergraphs (see \cite{EGMSTZ18,FKL19-1,FKL18,FKL19-2,GMSTV18}). Most of them focus on the size of a largest $r$-uniform connected $n$-vertex hypergraph with no Berge cycle of length at least $k$, among them, the best known result for the size of a largest  connected $n$-vertex $r$-graph with no Berge path of length $k$ was given by F\"uredi, Kostochka, and Luo~\cite{FKL19-2}.
We first define some constructions of extremal hypergraphs.
Let $S_r(n,k)$ be the $r$-graph on vertex set $A\cup B$ with $|A|=k$ and $|B|=n-k$, and edge set
$$E=\{e:e\subset A\cup B \mbox{ with } |e|=r \mbox{ and } |e\cap B|\le 1\}\mbox{;}$$
let  $S'_r(n,k)$ be the $r$-graph obtained from $S_r(n,k)$ by removing all the edges contained in $A$, i.e.
$$E(S'_r(n,k))=\{e:e\subset A\cup B \mbox{ with } |e|=r \mbox{ , } |e\cap B|=1\}.$$
Let $K_{k+1}^r$ be the complete $r$-graph on $k+1$ vertices and let $S(sK_{k+1}^{r},1)$ be the $r$-graph on $sk+1$ vertices consisting of $s$ copies of $K_{k+1}^r$ that intersect in exactly one common vertex called the center of $S(sK_{k+1}^{r},1)$.
Let $H_1$ and $H_2$ be two hypergraphs. We write $H_1\subseteq H_2$ for $H_1$ is a subgraph of $H_2$ and write $H_1\cong H_2$ for $H_1$ is isomorphic to $H_2$.
From the definitions, one can directly check  that the following proposition holds.
\begin{prop}\label{PROP: Extremal}
Suppose $k\ge r\ge 2$, $s\ge 2$, and $n\ge 2k+1$. Then $$\delta_1(S'_r(n,k))=\delta_1(S_r(n,k))=\delta_1(S(sK_{k+1}^r, 1))=\binom{k}{r-1},$$
$$e(S_r(n,k))=e(S'_r(n,k))+{k\choose r}=(n-k){k\choose r-1}+{k\choose r}, \,\, e(sK_{k+1}^r,1)=s{k+1\choose r},$$
and
a longest Berge path in $S'_r(n,k)$, $S_r(n,k)$ and  $S(sK_{k+1}^{r},1)$ has length $2k$.
\end{prop}

\begin{thm}[Theorem 18 in~\cite{FKL19-2}]\label{THM: FKL19-2}
 Let $k>4r>12$ and suppose $n$ is larger than a proper function of $k$ and $r$. If $H$ is an $n$-vertex connected $r$-graph with no Berge path of length $k$, then $$e(H)\le \left(n-\left\lceil\frac{k+1}2\right\rceil\right){\lfloor\frac{k-1}2\rfloor\choose r-1}+{\lceil\frac{k+1}2\rceil\choose r}.$$
 Furthermore, the $S_r(n, \lfloor\frac{k-1}2\rfloor)$ is an extremal hypergraph.
\end{thm}
In this paper, we extend  Theorem~\ref{THM: Dirac} to a hypergraph version for Berge paths, which is also a minimum degree version of Theorem~\ref{THM: FKL19-2}.

\begin{thm}\label{THM: Main1}
Given positive integers $n,k$ and $r$, let $H$ be a connected $n$-vertex $r$-graph with no Berge path of length $2k+1$.

(1) If  $k> r\ge 4$ and $n>2k+1$, then $\delta_1(H)\le\binom{k}{r-1}$. Furthermore, the equality holds if and only if  $S'_r(n,k)\subseteq H\subseteq S_r(n,k)$ or $H\cong S(sK_{k+1}^{(r)},1)$.

(2) If $k\ge r\ge 2$ and $n>2k(r-1)$, then $\delta_1(H)\le \binom{k}{r-1}$.

\end{thm}


\noindent{\bf Remarks:} (1) Proposition~\ref{PROP: Extremal} implies that the minimum degree threshold in Theorem~\ref{THM: Main1} is optimal.
For $n>2k+1>k>r\ge 4$, we characterize all extremal graphs.

(2) From Proposition~\ref{PROP: Extremal}, we know also that the sizes of the extremal graphs in Theorem~\ref{THM: Main1} are no more than the one in Theorem~\ref{THM: FKL19-2}. It is surprise for us that $S(sK_{k+1}^r, 1)$ is also an extremal graph in Theorem~\ref{THM: Main1}, which has different type of structure from $S_r(n,k)$ and $S'_r(n,k)$.


A Berge cycle $C$ in a hypergraph $H$ is called a Berge Hamiltonian cycle if $K(C)=V(H)$.
The Dirac-type results for Berge Hamiltonicity of hypergraphs have been studied in literatures.
\begin{thm}\label{THM: Main-Cor}

(a) (Bermond et al.~\cite{BGHS76}) Let $H$ be an $r$-uniform hypergraph on $n\ge r+1$ vertices. If
$\delta_1(H)\ge {{n-2}\choose {r-1}}+r-1$, then $H$ contains a hamiltonian Berge cycle.

(b) (Clemens et al.~\cite{CEP16}) Let $r\ge 3$ and let $H$ be an $r$-uniform hypergraph on $n>2r-2$ vertices. If
$\delta_1(H)\ge {{\lceil\frac n2\rceil-1}\choose {r-1}}+n-1$, then $H$ contains a hamiltonian Berge cycle.

(c) (Coulson, Perarnau~\cite{CP20}) Let $H$ be an $r$-uniform hypergraph on $n$ vertices, and suppose that $r=o(\sqrt n)$. If
$\delta_1(H)\ge {{\lceil\frac n2\rceil-1}\choose {r-1}}$, then $H$ contains a hamiltonian Berge cycle.

\end{thm}

Clearly, (b) and (c) are improvements of (a), as mentioned by the authors of~\cite{CEP16}, the term $n-1$ in (b) can be reduced, (c) eliminated the term $n-1$ from (b), but which requires that $r$ is much smaller than $n$.
As an application of Theorem~\ref{THM: Main1}, we reduce the term $n-1$ to $\lceil\frac{n-1}{2}\rceil$ for $r\ge 4$ and $n\ge 2r+4$.
\begin{thm}\label{THM: Main3}
For $r\ge 4$ and $n\ge 2r+4$, let $H$ be an $r$-graph with $\delta_1(H)>\binom{\lfloor\frac{n-1}{2}\rfloor}{r-1}+\lceil\frac{n-1}{2}\rceil$, then $H$ contains a hamiltonian Berge cycle.
\end{thm}

The rest of this article is arranged as follows. In Section 2, we give some useful lemmas. We prove Theorems~\ref{THM: Main1} and \ref{THM: Main3} in Sections 3 and 4, respectively. Some discussions and remarks will be given in the last section.

\section{Preliminaries and lemmas}

For an $r$-graph $H$, we use $\ell(H)$ (resp. $c(H)$) denote the length of a longest Berge path (resp. Berge cycle) contained in $H$.

\begin{lem}\label{LEM: cycle1}
Let $H$ be a connected $n$-vertex $r$-graph. If $\ell(H)+1\le c(H)$ then $n=c(H)$.
\end{lem}
\begin{proof}
Since every Berge cycle of length $t+1$ has a Berge path of length $t$ as its subgraph, $c(H)\le \ell(H)+1$, which implies that $\ell(H)+1=c(H)$.
Choose a longest Berge cycle $C$ in $H$ with $K(C)=\{v_1,v_2,...,v_{t}\}$ and $E(C)=\{e_1,...,e_{t}\}$, i.e. $t=c(H)$.
First, we claim that $V(C)=K(C)$. Otherwise, without loss of generality, suppose there is a vertex $v_0\in e_{t}$ and $v_0\notin K(C)$. Then we can find a Berge path $P$ of length $t$ with $K(P)=\{v_0,v_1,...,v_t\}$ and $E(P)=\{e_t,e_1,e_2,...,e_{t-1}\}$, which contradicts to $\ell(H)=t-1$.
Since $n=|V(H)|\ge|V(C)|=t$, we only need to prove that $n\le t$. Suppose $n>t$. Since  $H$ is connected, there must be a vertex $u\in V(H)\setminus V(C)$ and an edge $e\in E(H)\setminus E(C)$ such that $u\in e$ and $e\cap V(C)\neq\emptyset$. Without loss of generality, assume $v_t\in e$. Then again we can find a Berge path $P$ of length $t$ with $K(P)=\{v_1,...,v_t, u\}$ and $E(P)=\{e_1,e_2,...,e_{t-1},e\}$, which is a contradiction.
\end{proof}

\begin{lem}\label{LEM: cycle2}
Let $k\ge r\ge 2$ and $n>2k(r-1)$ and let $H$ be a connected $n$-vertex $r$-graph with $\delta_1(H)>\binom{k}{r-1}$. If $\ell(H)\le c(H)$ then $c(H)\ge 2k+1$.
\end{lem}
\begin{proof}
If $\ell(H)\le c(H)-1$, by Lemma~\ref{LEM: cycle1}, we have $c(H)=n\ge 2k(r-1)+1\ge 2k+1$. Now assume $\ell(H)=c(H)=t$.
Suppose to the contrary that $t\le 2k$. Pick a Berge cycle $C$ of length $t$ in $H$ with $K(C)=\{v_1,v_2,...,v_{t}\}$ and $E(C)=\{e_1,...,e_{t}\}$.\\
Since $$|V(C)|=\left|\bigcup_{i=1}^{t}e_i\right|\le\sum_{i=1}^{t}|e_i\setminus\{v_i,v_{i+1}\}|+|K(C)|\le 2k(r-1),$$
where the indices take modulo $t$. Since $H$ is connected and $n>2k(r-1)$, there is a vertex $u\in V(H)\setminus V(C)$ and an edge $e_0\in E(H)\backslash E(C)$ such that $u\in e_0$ and $e_0\cap V(C)\neq\emptyset$.
\begin{claim}\label{CLAIM: c1}
For any edge $e$ with $u\in e$, $e\setminus\{u\}\subseteq  K(C)$. 
\end{claim}
Suppose to the contrary that there is a vertex $w\in e\setminus \{u\}$ but $w\notin K(C)$. If $e\cap V(C)\not=\emptyset$, without loss of generality, assume  $w\in e\cap e_t$. Then the path $P$ with $K(P)=\{v_1,v_2,...,v_t,w,u\}$ and $E(P)=\{e_1,e_2,...,e_t,e\}$ is a Berge path of length $t+1$, a contradiction.
Now assume $e\cap V(C)=\emptyset$. Then $e\neq e_0$. Without loss of generality, suppose $v_t\in e_0$. Then the path $P$  with $K(P)=\{v_1,...,v_t,u,w\}$ and $E(P)=\{e_1,...,e_{t-1},e_0,e\}$ is a Berge path of length $t+1$, a contradiction too.



Let $N=\cup_{u\in e}(e\setminus\{u\})$. By Claim~\ref{CLAIM: c1}, $N\subseteq K(C)$. We say $v_i,v_j\in N$ are equivalent if and only if for each $v_p\in \{v_i,v_{i+1}, \ldots, v_j\}$ or $v_p\in\{v_j, v_{j+1}, \ldots, v_i\}$, $v_p\in N$, where the indices take modulo $t$. Clearly, we can partition $N$ into equivalent classes by the equivalent relation. Let $\mathcal{C}$ be the set of equivalent classes of $N$.
Let $n_i$ and $n_{\ge i}$ be the numbers of the equivalent classes of size $i$ and at least $i$, respectively.
Since for each class $M=\{v_i,...,v_j\}\in \mathcal{C}$ we have $v_{j+1}\notin N$. So $|N|\le t-|\mathcal{C}|=t-n_1-n_{\ge 2}$. Also we have $|N|=\sum\limits_{M\in\mathcal{C}}|M|\ge n_1+2n_{\ge 2}$. Therefore, $n_1\le\frac{t}{2}\le k$ and $n_{\ge2}\le\frac{1}{3}(t-2n_1)\le\frac{2}{3}(k-n_1)$.

Now choose $M\in\mathcal{C}$ with $|M|\ge 2$. Assume $M=\{v_i, v_{i+1}, \ldots, v_j\}$. Choose $e$ such that $u, v_i\in e$. We claim that $M\subseteq e$ and for any edge $e'\neq e$ with $u\in e'$, $e'\cap M=\emptyset$. In fact, we show that for any $p\in \{i, i+1,\cdots, j-1\}$ if $u, v_p\in e$ then for any edge $e'\neq e$ with $u\in e'$, $v_{p+1}\notin e'$. Otherwise, the cycle $C$ with $K(C)=\{v_1,\ldots, v_p, u, v_{p+1},\ldots, v_t\}$ and  $E(C)=\{e_1,\ldots,e_{p-1},e,e', e_{p+1}, \ldots, e_t\}$ is a Berge cycle of length $t+1$, a contradiction. Therefore, $v_{p+1}\in e$ by the definition of $N$. This implies the claim.
By the claim, an equivalent class $M$ with $|M|\ge 2$ is contained in only one edge $e$ with $u\in e$. So there are at most $n_{\ge 2}$ edges $e$ such that $u\in e$ and $e$ contains at least one equivalent class of size at least $2$.
Hence,
\begin{equation}\label{EQN: e1}
\binom{k}{r-1}<d_H(u)\le n_{\ge 2}+\binom{n_1}{r-1}\le\frac{2k}{3}+\binom{n_1}{r-1}-\frac{2n_1}{3}\mbox{.}
\end{equation}
Let $f(x)=\binom{x}{r-1}-\frac{2x}{3}$. By inequality (1), we have $f(k)<f(n_1)$. But, by the convexity of  $f(x)$ on $[0,k]$ and $n_1\in [0,k]$, we have  $f(n_1)\le\max\{f(0),f(k)\}=f(k)$, a contradiction.
\end{proof}

In the following, we give some properties of the longest Berge paths in a hypergraph.
Let $P$ be a Berge path with $K(P)=\{v_0,v_1,...,v_t\}$ and $E(P)=\{e_1,...,e_t\}$ in an $r$-graph $H$. For $a\in[0,t]$, we define
$E^O_a(P)=\{ e\in E(H)\setminus E(P) : v_a\in e\}$,
$$K_a(P)=\{v_i\in K(P)\setminus\{v_a\} : \mbox{there exists  } e\in E_a^O(P)\mbox{ with } v_i\in e\},$$
$\kappa_a(P)=\{i : v_i\in K_a(P)\}$, $E^I_a(P)=\{ e_i\in E(P) : v_a\in e_i\}$ and $\varepsilon^i_a(P)=\{ i : e_i\in E^I_a(P)\}$.
If the path $P$ is clear from the context, write $K_a, \kappa_a, E^I_a, \varepsilon^i_a, E^O_a$ for $K_a(P), \kappa_a(P), E^I_a(P), \varepsilon^i_a(P), E^O_a(P)$ for short.
Similarly, for $v\in V(H)\setminus K(P)$, define $E^O_v(P)=\{ e\in E(H)\setminus E(P) : v\in e\}$,
$$K_v(P)=\{ v_i: \mbox{ there exists } e\in E^O_v(P)\mbox{ with } v_i\in e\},$$
$\kappa_v(P)=\{ i : v_i\in K_v(P)\}$ and $E^I_v(P)=\{ e_i\in E(P) : v\in e_i\}$,  $\varepsilon^i_v(P)=\{ i : e_i\in E^I_v(P)\}$. Also $P$ will be omitted if the path $P$ is clear from the context.
A Berge path $P$ is called {\it extendible } if we can get a longer Berge path from $P$ by removing $s$ edges from $P$ and adding at least $s+1$ new edges.
Otherwise, we call $P$ is {\it non-extendible}. Clearly, a longest Berge path in a hypergraph is non-extendible.
For a family $A$ of sets and a set $B$, let $A+B=\{a\cup B : a\in A\}$ and $A-B=\{a-B : a\in A\}$.
 For a set $A$ of integers and an integer $b$, let $A+b=\{a+b : a\in A\}$ and $A-b=\{a-b : a\in A\}$.

The following is a simple observation.
\begin{obs}\label{PROP: p0}
Let $P$ be a Berge path in $r$-graph $H$ with $E(P)=\{e_1,...,e_t\}$ and $K(P)=\{v_0,v_1,...,v_t\}$.
If $K_v(P)\cap\{v_0, v_t\}\not=\emptyset$ for some $v\notin K(P)$ then $P$ is extendible.

\end{obs}


\begin{cor}\label{PROP: p1}
Let $P$ be a longest Berge path in $r$-graph $H$ with $E(P)=\{e_1,...,e_t\}$ and $K(P)=\{v_0,v_1,...,v_t\}$. Then the following statements hold.

(a) For each edge $e\in E_a^O(P)$, we have $e\setminus\{v_a\}\subseteq K_a(P)$ for $a\in\{0,t\}$.

(b) $d_H(v_a)\le \binom{|K_a|}{r-1}+|E^I_a|$ for $a\in\{0,t\}$. Moreover, the equality holds if and only if $N_H(v_a)=\binom{K_a}{r-1}\cup (E^I_a-\{v_a\})$ and $\binom{K_a}{r-1}\cap (E^I_a-\{v_a\})=\emptyset$.
\end{cor}
\begin{proof}
(a) is clearly true. Otherwise, $P$ will be an extendible Berge path by Observation~\ref{PROP: p0}.


(b) By (a), for any $e\in E^O_a(P)$, we have $e\setminus\{v_a\}\subseteq K_a(P)$, $a\in\{0,t\}$. So $N_H(v_a)\subseteq \binom{K_a}{r-1}\cup (E^I_a-\{v_a\})$. Thus $d_H(v_a)\le\binom{|K_a|}{r-1}+|E^I_a|$.
The equality holds if and only if $N_H(v_a)$ is the disjoint union of $\binom{K_a}{r-1}$ and $E^I_a-\{v_a\}$.
\end{proof}

The following proposition plays an important role in the proof of the main theorem.
\begin{prop}\label{PROP: set}
Suppose $k\ge r\ge 2$, $t\le 2k$ and $H$ is a connected $n$-vertex $r$-graph with $\ell(H)=t$. Let $P$ be a longest Berge path with $E(P)=\{e_1,...,e_t\}$ and $K(P)=\{v_0,v_1,...,v_t\}$. The following properties hold.

\begin{itemize}
\item[(1)] For $n>t+1$ and $v\notin K(P)$, we have

\begin{itemize}
\item[(1.1)] $\kappa_0, \kappa_t, \kappa_v\subseteq [1,t-1]$, $1\in \varepsilon^i_0$ and $t\in \varepsilon^i_t$;

\item[(1.2)] $(\kappa_0-1)\cap \kappa_t=\emptyset$, $(\kappa_0-1)\cap\kappa_v=\emptyset$, and $(\kappa_{t}+1)\cap\kappa_v =\emptyset$;

\item[(1.3)] $(\varepsilon^i_0-1)\cap \kappa_t=\varepsilon^i_t\cap \kappa_0=\emptyset$;

\item[(1.4)] for $i\in \kappa_0$ and $j\in\kappa_t$, if $i\le j$ then $E_{i-1}^O\cap E_{j+1}^O=\emptyset$, and if $i>j$ then $E_{i-1}^O\cap E_{j-1}^O=\emptyset$ and $E_{i+1}^O\cap E_{j+1}^O=\emptyset$;

\item[(1.5)] for $i\in \kappa_0$ and $j\in\kappa_t$, if $i\le j$ then $v_{j+1}\notin e_{i}$ and $v_{i-1}\notin e_{j+1}$; if $i>j$ then $v_{j-1}\notin e_{i}$ and $v_{j+1}\notin e_{i+1}$.

\item[(1.6)] for $a\in\{0, t\}$, $\kappa_a\cap(\kappa_a+1)\cap \varepsilon_0^i\cap\varepsilon_v^i=\emptyset$.
\end{itemize}

\end{itemize}

\begin{itemize}
\item[(2)]
If $n>2k(r-1)$ and $\delta_1(H)>\binom{k}{r-1}$, then
\begin{itemize}
\item[(2.1)]
$(\varepsilon^i_0-1)\cap \varepsilon^i_t=\emptyset$;

\item[(2.2)]
 $(\varepsilon^i_0-2)\cap \kappa_t=(\varepsilon^i_t+1)\cap \kappa_0=\emptyset$.
 \end{itemize}
 \end{itemize}
\end{prop}
\begin{proof}
(1.1) If $t\in \kappa_0$, then there is an edge $e\in E(H)\setminus E(P)$ with $v_0, v_t\in e$. Thus $C=\{e_1,e_2,...,e_t,e\}$ is a Berge cycle of length $t+1$, which implies that $\ell(H)+1\le c(H)$. By Lemma~\ref{LEM: cycle1}, $n=t+1$, a contradiction to $n>t+1$. So $t\notin \kappa_0$. Similarly, $0\notin \kappa_t$. Therefore, $\kappa_0, \kappa_t\subseteq [1,t-1]$.
$\kappa_v\subseteq[1,t-1]$ follows directly from Observation~\ref{PROP: p0}.

By the definitions of $\varepsilon^i_0$ and $\varepsilon^i_t$, we have  $1\in \varepsilon^i_0$ and $t\in \varepsilon^i_t$.

(1.2) Suppose to the contrary that there exists $i\in (\kappa_0-1)\cap \kappa_t\subseteq [1,t-2]$. Then there are edges $e\in E^O_0(P)$ with $v_{i+1}\in e$ and $f\in E^O_t(P)$ with $v_i\in f$.
We claim that  $e\not=f$. Otherwise, we have $t\in \kappa_0$, a contradiction to (1.1). So the cycle $C$ with $E(C)=\{e_1, \ldots, e_i, f, e_t, \ldots, e_{i+2}, e\}$ and $K(C)=\{v_0,  \ldots, v_i, v_t,  \ldots, v_{i+1}\}$ is a Berge cycle of length $t+1$, which again implies that $\ell(H)+1\le c(H)$. By Lemma~\ref{LEM: cycle1}, $n=t+1$, a contradiction too.

Now suppose to the contrary that there is an $i\in (\kappa_0-1)\cap \kappa_v$. Then there are edges $e\in E^O_0(P)$ with $v_{i+1}\in e$ and $f\in E^O_v(P)$ with $v_i\in f$. Note that $e\neq f$ by Observation~\ref{PROP: p0}. So $P'$ with $K(P')=\{v, v_i, v_{i-1}, \ldots, v_0, v_{i+1}, \ldots, v_{2k}\}$ and $E(P)=\{f, e_i, \ldots, e_1, e, e_{i+2}, \ldots, e_{2k}\}$ is a Berge path of length $2k+1$, a contradiction. Similarly, we can show $\kappa_v\cap (\kappa_{t}+1)=\emptyset$.

(1.3) Suppose to the contrary that there is an $i\in(\varepsilon^i_0-1)\cap \kappa_t\subseteq [1,t-1]$. Then $v_0\in e_{i+1}$ and there exists an edge $e\in E^O_t(P)$ with $v_t, v_i\in e$. So the cycle $C$ with $E(C)=\{e_1,\ldots,e_i,e,e_t, \ldots, e_{i+2}, e_{i+1}\}$ and $K(C)=\{v_0, \ldots, v_i, v_t, \ldots, v_{i+1}\}$ is a Berge cycle of length $t+1$. We get  a contradiction with a same reason as in Case (1.1). With similar arguments, we have $\varepsilon^i_t\cap \kappa_0=\emptyset$.

(1.4) Let $e\in E^O_0(P)$ with $v_{i}\in e$ and $f\in E^O_t(P)$ with $v_{j}\in f$. If $i\le j$, suppose that there exists $g\in E_{i-1}^O\cap E_{j+1}^O$, i.e. $g\notin E(P)$ and $v_{i-1}, v_{j+1}\in g$.
We claim that  $e, f, g$ are pairwise different. In fact, if $e=f$ then $t\in \kappa_0$, a contradiction to (1.1); if $e=g$ then $j+1\in\kappa_0$, a contradiction to $(\kappa_0-1)\cap \kappa_t=\emptyset$; if $f=g$ then $i-1\in\kappa_t$, again a contradiction to $(\kappa_0-1)\cap \kappa_t=\emptyset$.    So the cycle $C$ with $E(C)=\{e_1, \ldots, e_{i-1}, g, e_{j+2},\ldots, e_t, f, e_{j}, \ldots, e_{i+1}, e\}$ and $K(C)=\{v_0,  \ldots, v_{i-1}, v_{j+1}, \ldots,  v_t, v_{j}, \ldots, v_{i}\}$ is a Berge cycle of length $t+1$, which again implies that $\ell(H)+1\le c(H)$. By Lemma~\ref{LEM: cycle1}, $n=t+1$, a contradiction too. With similar discussion, we have  $E_{i-1}^O\cap E_{j-1}^O=\emptyset$ for $i>j$.

(1.5) It can be proved similarly as (1.4), just replace $g$ by $e_{i}$.

(1.6) If not, suppose $j\in\kappa_a\cap(\kappa_a+1)\cap \varepsilon_0^i\cap\varepsilon_v^i$, then there is edges $e\notin E(P)$ with $v_{j-1}, v_j\in e$ and $v_a, v\in e_j$. Without loss of generality, assume $a=0$. So  $P'$ with $E(P')=\{e_j, e_1, \ldots, e_{j-1}, e, e_{j+1}, \ldots, e_{t}\}$ and $K(P')=\{v,v_0,v_1, \ldots, v_t\}$ is a Berge path of length $t+1$, a contradiction.

\medskip
Now suppose $n>2k(r-1)$ and $\delta_1(H)>\binom{k}{r-1}$.

(2.1) Suppose to the contrary that there is an $i\in (\varepsilon^i_0-1)\cap \varepsilon^i_t\subseteq[1,t-1]$. Then $v_0\in e_{i+1}$ and $v_t\in e_i$. So the cycle $C$ with $E(C)=\{e_1, \ldots, e_{i}, e_t, e_{t-1},\ldots, e_{i+1}\}$ and $K(C)=\{v_0, \ldots, v_{i-1}, v_t, v_{t-1}, \ldots, v_{i+1}\}$ is a Berge cycle of length $t$, which implies that $\ell(H)\le c(H)=t$. By Lemma~\ref{LEM: cycle2}, $t\ge 2k+1$, a contradiction.

(2.2) Suppose to the contrary that there is an $i\in(\varepsilon^i_0-2)\cap \kappa_t\subseteq [1,t-2]$. Then $v_0\in e_{i+2}$ and there exists an edge $e\in E^O_t(P)$ with $v_t,v_i\in e$. So the cycle $C$  with $E(C)=\{e_1, \ldots, e_i, e, e_t,  \ldots, e_{i+3}, e_{i+2}\}$ and $K(C)=\{v_0, \ldots, v_i, v_t,  \ldots, v_{i+2}\}$ is a Berge cycle of length $t$, a contradiction again. Similarly, we have $(\varepsilon^i_t+1)\cap \kappa_0=\emptyset$.
\end{proof}


\noindent{\bf Remark of Proposition~\ref{PROP: set}:} By the proof of (1), we directly have (1.1), (1.2) and (1.3) still hold for $n=t+1$ provided that $H$ does not contain a Berge Hamiltonian cycle.

\section{The Proof of Theorem~\ref{THM: Main1}}
\subsection{The Proof of Theorems~\ref{THM: Main1} (1)}
We first prove a weak version.
\begin{thm}\label{THM: weak}
Let $k>r\ge 4$ and $n>2k+1$. If $H$ is a connected $n$-vertex $r$-graph  with $\delta_1(H)\ge\binom{k}{r-1}$ then $\ell(H)\ge 2k$. Moreover, if $\ell(H)=2k$ then for any longest Berge path $P$ with $E(P)=\{e_1, e_2,\ldots, e_{2k}\}$ and $K(P)=\{v_0, v_1,\ldots, v_{2k}\}$, we have $|K_0(P)|=|K_{2k}(P)|=k$.
\end{thm}
\begin{proof}
Let $P$ be a longest Berge path in $H$ with $E(P)=\{e_1,...,e_{t}\}$ and $K(P)=\{v_0,v_1,...,v_{t}\}$.
By (b) of Corollary~\ref{PROP: p1}, $d_H(v_a)\le \binom{|K_a|}{r-1}+|E^I_a|$ for $a\in\{0,t\}$.
We prove by contradiction. Suppose to the contrary that $t=\ell(H)\le 2k$. By (1.1) and (1.2) of Proposition~\ref{PROP: set},
\begin{equation*}
|K_0|+|K_t|=|\kappa_0-1|+|\kappa_t|=|(\kappa_0-1)\cup \kappa_t|\le|[0,t-1]|=t\le 2k.
\end{equation*}
Similarly, by (1.3) of Proposition~\ref{PROP: set}, we get
\begin{equation}\label{EQU: e1}
|E^I_0|=|\varepsilon^i_0|=|\varepsilon^i_0-1|\le t-|\kappa_t|\le 2k-|\kappa_t|
\end{equation} and $|E^I_t|=|\varepsilon^i_t|\le t-|\kappa_0|\le 2k-|\kappa_0|$.
Without loss of generality, assume $|K_0|\le|K_t|$. Then $|K_0|\le k$ and
\begin{eqnarray}\label{EQN:e3}
\binom{k}{r-1}&\le& d_H(v_0)\le\binom{|K_0|}{r-1}+|E^I_0|\nonumber\\
              &\le&\binom{|K_0|}{r-1}+2k-|K_t|\nonumber\\
              &\le&\binom{|K_0|}{r-1}+2k-|K_0|\mbox{.}
\end{eqnarray}

If $|K_0|=k$ then  $|K_t|=k$ and $t=2k$, we are done.
So now we assume $|K_0|<k$.

\medskip
\noindent\textbf{Case 1.}  $(k,r)\neq(5,4)$.

By the convexity of the function $f(x)=\binom{x}{r-1}+2k-x$ for $x\in [0,k-1]$, we have
\begin{eqnarray}\label{EQN: e2}
\binom{k}{r-1}&\le& \binom{|K_0|}{r-1}+2k-|K_0|\nonumber\\
              &\le&\max\{f(0),f(k-1)\}=\binom{k-1}{r-1}+k+1\nonumber\\
                            &=&\binom{k}{r-1}+\left(k+1-\binom{k-1}{r-2}\right).
\end{eqnarray}
Thus $k+1-\binom{k-1}{r-2}\ge 0$. But $k+1-\binom{k-1}{r-2}\le k+1-\binom{k-1}{2}=k(5-k)/2<0$ when $(k,r)\not=(5,4)$, a contradiction.

\noindent\textbf{Case 2.}  $(k,r)=(5,4)$.

Then all the equalities hold in the inequalities~(\ref{EQU: e1}), (\ref{EQN:e3}) and (\ref{EQN: e2}), that is $d_H(v_a)=\binom{5}{3}=10$, $t=2k=10$, $|K_t|=|K_0|=k-1=4$, and  by (b) of Corollary~\ref{PROP: p1},  $N_H(v_a)=\binom{K_a}{3}\cup (E^I_a-\{v_a\})$ for $a\in \{0, 10\}$.
So $|E^I_a|=|E^I_{a}-\{v_a\}|=d_H(v_a)-|\binom{K_a}{3}|=6$. By (1.1), (1.2) and (1.3) of Proposition~\ref{PROP: set}, $\kappa_0, \kappa_{10}\subseteq [1,9]$, $(\kappa_0-1)\cap \kappa_{10}=\emptyset$, and $(\varepsilon^i_0-1)\cap \kappa_{10}=\varepsilon^i_{10}\cap \kappa_0=\emptyset$. Therefore,  $(\varepsilon^i_0-1)\cup \kappa_{10}=(\varepsilon^i_{10}-1)\cup (\kappa_0-1)=[0,9]$. So  $\kappa_0-1\subseteq [0,9]\setminus \kappa_{10}=\varepsilon^i_0-1$.

\begin{claim}\label{CLA: c2}
$\kappa_0\cap(\kappa_0-1)=\emptyset$ and $\kappa_0\subseteq [2,9]$. Similarly, $\kappa_{10}\cap(\kappa_{10}-1)=\emptyset$ and $\kappa_{10}\in[1,8]$.
\end{claim}
If $\kappa_0\cap (\kappa_0-1)\neq\emptyset$, say $i\in \kappa_0\cap(\kappa_0-1)$, then $i-1,i\in \kappa_0-1\subseteq\varepsilon^i_0-1$. So $i, i+1\in \kappa_0\subseteq \varepsilon^i_0$ and $i+1\not\in \varepsilon^i_{10}$. Let $\kappa_0=\{i, i+1, j_1, j_2\}$. Then $e'=\{v_0,v_i,v_{i+1},v_{j_1}\}\in E_0^O(P)$. Since $\binom{K_0}{3}\cap (E^I_0-\{v_0\})=\emptyset$, $e_{i+1}\setminus\{v_0\}\not\in \binom{K_0}{3}$. Assume $e_{i+1}=\{v_0,v_i,v_{i+1},u\}$. Then $u\not= v_{j_1}, v_{j_2}, v_{10}$. If $u\notin K(P)$ then the path $P'$ with $K(P')=\{u,v_0,...,v_{10}\}$ and $E(P')=\{e_{i+1},e_1, \ldots, e_{i}, e', e_{i+2}, \ldots, e_{10}\}$ is a Berge path of length $11$, a contradiction. So we assume $u\in K(P)$. Then the path $P''$ with $K(P'')=K(P)$ and $E(P'')=(E(P)\setminus\{e_{i+1}\})\cup\{e'\}$ is also a Berge path of length $10$. But $K_0(P'')=\{v_i, v_{i+1},  v_{j_1}, v_{j_2}, u\}$ and $E^I_{10}(P'')=E^I_{10}(P)$. By (1.3) of Proposition~\ref{PROP: set}, $|K_0(P'')\cup E^I_{10}(P'')|=|K_0(P'')|+|E^I_{10}(P'')|=11$, but $|K_0(P'')\cup E^I_{10}(P'')|\le |[1,10]|=10$, this is a contradiction.

Now suppose $1\in \kappa_0$. Denote $\kappa_0=\{1, j_1, j_2, j_3\}$. Then $e'=\{v_0, v_1, v_{j_1}, v_{j_2}\}\in E^O_0(P)$. Assume $e_1=\{v_0, v_1, u,v\}$. With a similar discussion with the above case, if $e_1\nsubseteq K(P)$ then we have a Berge path of length 11 by adding $e'$ to $P$; if $e_1\subseteq K(P)$ then we obtain a Berge path $Q$ of length 10 with $|K_0(Q)|>|K_0(P)|$ and $E^I_{10}(Q)=E^I_{10}(P)$ by replacing $e_1$ by $e'$. In each case, we get a contradiction.
Therefore, $\kappa_0\cap(\kappa_0-1)=\emptyset$ and $\kappa_0\subseteq [2,9]$. Similarly, we have $\kappa_{10}\cap(\kappa_{10}-1)=\emptyset$ and $\kappa_{10}\subseteq [1,8]$.

By Claim~\ref{CLA: c2}, $\kappa_0-1, \kappa_{10}\subseteq [1,8]$ and $\kappa_a-1$ contains no consecutive integers for $a\in\{0,1\}$. Since $(\kappa_0-1)\cap \kappa_{10}=\emptyset$ and $|\kappa_0|=|\kappa_{10}|=4$, we have $(\kappa_0-1)\cup \kappa_{10}=[1,8]$. This forces that $\kappa_0=\kappa_{10}=\{2,4,6,8\}$. So $\varepsilon^i_0=[0,9]\setminus \kappa_{10}+1=\{1,2,4,6,8,10\}$ and $\varepsilon^i_{10}=[0,9]\setminus(\kappa_0-1)+1=\{1,3,5,7,9,10\}$. This implies that $\{v_0, v_1, v_2\}\subseteq e_2$ and $\{v_1, v_{10}\}\subseteq e_1$. Set $e'=\{v_0, v_2, v_4, v_6\}$. Then $e'\in E^O_0(P)$ because of $\binom{K_0}{3}\subseteq N_H(v_0)$.  Therefore, the cycle $C$  with $E(C)=\{e', e_3, \ldots, e_{10}, e_1, e_2\}$ and $K(C)=\{v_0, v_2, v_3, \ldots, v_{10}, v_1\}$ is a Berge cycle of length $11$. Since $c(H)\le \ell(H)+1=11$, we have $c(H)=11$. By Lemma~\ref{LEM: cycle1}, $n=c(H)=11$, a contradiction to $n>2k+1=11$.
\end{proof}

The following theorem characterize the extremal graphs in Theorem~\ref{THM: Main1} (1).
\begin{thm}\label{THM: 2k}
For $k\ge r\ge 2$, $n>2k+1$ and $d\in\{\binom{k}{r-1},\binom{k}{r-1}+1\}$, let $H$ be a connected $n$-vertex $r$-graph  with $\delta_1(H)\ge d$. If
$\ell(H)=2k$ and for any longest Berge path $P$ in $H$, we have $|\kappa_0(P)|=|\kappa_{2k}(P)|=k$, then $d=\binom{k}{r-1}$ and either $S'_r(n,k)\subseteq H\subseteq S_r(n,k)$ or $H\cong S(sK_{k+1}^{(r)},1)$ with $n=sk+1$.
\end{thm}
\begin{proof}
Let $P$ be a a longest Berge path in $H$ with $E(P)=\{e_1, \ldots, e_{2k}\}$ and $K(P)=\{v_0, v_1, \ldots, v_{2k}\}$.
By (1.1) and (1.2) of  Proposition~\ref{PROP: set},  $\kappa_0-1\subseteq [0, 2k-2], \,  \kappa_{2k}\subseteq[1,2k-1]$ and $(\kappa_0-1)\cap \kappa_{2k}=\emptyset$. By Theorem~\ref{THM: weak},  $|\kappa_0-1|=|\kappa_{2k}|=k$. Thus $(\kappa_0-1)\cup \kappa_{2k}=[0,2k-1]$ and $0\in \kappa_0-1, \, 2k-1\in\kappa_{2k}$.

\begin{claim}\label{CLA: c4}
$\kappa_v\cap (\kappa_v-1)=\emptyset$ for any $v\notin K(P)$. Furthermore, $K(P)=V(P)$.
\end{claim}
Suppose to the contrary that there is an $i\in\kappa_v\cap (\kappa_v-1)$ for some $v\notin K(P)$. Then $i,i+1\in \kappa_v$. Since $i\in[0,2k-1]=(\kappa_0-1)\cup \kappa_{2k}$, we have either $i\in \kappa_{0}-1$ or $i\in \kappa_{2k}$. But this is impossible since  $\kappa_v\cap(\kappa_0-1)=\emptyset$ or $\kappa_v\cap(\kappa_{2k}+1)=\emptyset$ from  (1.2) of Proposition~\ref{PROP: set}. Now suppose there is a vertex $v\in V(P)\setminus K(P)$. Suppose $v\in e_{j+1}$. Then $v_j, v_{j+1}\in K_v(P)$, i.e. $j, j+1\in \kappa_v(P)$. This is impossible since $\kappa_v\cap (\kappa_v-1)=\emptyset$.

\medskip
\noindent\textbf{Case 1.} $\max(\kappa_0-1)>\min\kappa_{2k}$. 

Since $(\kappa_0-1)\cup \kappa_{2k}=[0,2k-1]$, we  can pick $1\le i_0\le 2k-3$ such that $i_0\in \kappa_{2k}$ and $i_0+1\in \kappa_0-1$. Thus $i_0\in \kappa_{2k}$ and $i_0+2\in \kappa_0$. Let  $e_0\in E^O_{0}(P)$ and $f_0\in E^O_{2k}(P)$ be the edges such that $v_0, v_{i_0+2}\in e_0$ and $v_{2k}, v_{i_0}\in f_0$.

Since $V(P)=K(P)$, $n>2k+1=|V(P)|$ and $H$ is connected, we have $V(H)\setminus V(P)\not=\emptyset$ and $E(H)\setminus E(P)\not=\emptyset$.
\begin{claim}\label{CLA: c3}
For any $v\in V(H)\setminus V(P)$ and any $e\in E(H)$ with $v\in e$, $e\setminus\{v\}\subseteq V(P)$.
\end{claim}
Let $v\in V(H)\setminus V(P)$. We first claim that: ($\ast$) there is no Berge path $Q$ of length at least two connecting $v$ and some vertex of $P$ such that $E(Q)\cap E(P)=\emptyset$ and $|K(Q)\cap V(P)|=1$. Suppose to the contrary that there is such a Berge path $Q$  connecting $v$ and some vertex $v_j\in V(P)$.  Then $|E(Q)|\ge 2$. Set $K(Q)=\{x, \ldots, v_j\}$ such that $K(Q)\cap V(P)=\{v_j\}$.
If $j\le i_0$, then  $P'$ with $K(P')=K(Q)\cup\{v_{j+1},...,v_{i_0}, v_{2k}, \ldots, v_{i_0+2}, v_0, \ldots, v_{j-1}\}$ and $E(P')=E(Q)\cup \{e_{j+1},\ldots, e_{i_0}, f_0, e_{2k}, \ldots,$ $ e_{i_0+3}, e_0, e_{1}, \ldots, e_{j-1}\}$ is a Berge path of at least $2k+1$, a contradiction. If $j=i_0+1$, then $P'$ with $K(P')=K(Q)\cup \{v_{i_0+2}, v_0, \ldots, v_{i_0}, v_{2k}, \ldots, v_{i_0+3}\}$ and $E(P')=E(Q)\cup \{e_{i_0+2}, e_0, e_1, \ldots, e_{i_0}, f_0, e_{2k}, \ldots, e_{i_0+4}\}$ is a Berge path of length at least $2k+2$, a contradiction. So assume $j\ge i_0+2$. Then $P'$ with $K(P')=K(Q)\cup\{v_{j-1},...,v_{i_0+2}, v_0, \ldots,$ $ v_{i_0}, v_{2k}, \ldots, v_{j+1}\}$ and $E(P')=E(Q)\cup \{e_{j-1},\ldots, e_{i_0+3}, e_0, e_1, \ldots, e_{i_0}, f_0, e_{2k}, \ldots, e_{j+2}\}$ is a Berge path of length at least $2k+1$, a contradiction too. ($\ast$) holds.
The claim ($\ast$) also implies that for any edge $e\in E(H)$ with $v\in e$, $e\cap V(P)\not=\emptyset$.
Now suppose  that there is a vertex $w\in e\setminus\{v\}$ with $w\notin V(P)$. Choose an edge $f$ with $w\in f$ ($f$ exists since $d_H(w)\ge d\ge 2$). So $f\cap V(P)\not=\emptyset$. Assume $v_j\in f\cap V(P)$.
Then $Q$ with $K(Q)=\{v,w, v_j\}$ and $E(Q)=\{e,f\}$ is a Berge path of length two connecting $v$ and $v_j$. Clearly, $E(Q)\cap E(P)=\emptyset$ and $|K(Q)\cap V(P)|=1$. By ($\ast$), this is impossible.

Now choose $v\notin K(P)$.
By Claim~\ref{CLA: c3}, we have $e\setminus\{v\}\subseteq K_v(P)$ for every $e$ with $v\in e$. Therefore,
$$d_H(v)\le\left|\binom{K_v(P)}{r-1}\right|=\binom{|\kappa_v(P)|}{r-1}\mbox{.}$$
Since $\kappa_v(P)\cap (\kappa_v(P)-1)=\emptyset$ and $\kappa_v(P)\subseteq[1,2k-1]$, we get $|\kappa_v(P)|\le k$ and the equality holds if and only if $\kappa_v(P)=\{1,3,5, \ldots, 2k-1\}$, which also implies that $\kappa_0=\kappa_{2k}=\{1,3,5,\ldots, 2k-1\}$ by (1.2) of Proposition~\ref{PROP: set}.
Since $d_H(v)\ge \delta_1(H)\ge\binom{k}{r-1}$, we have $|\kappa_v(P)|=k$. Therefore, $N_H(v)=\binom{K_v(P)}{r-1}=\binom{\{v_1,v_3,...,v_{2k-1}\}}{r-1}$ for all $v\in V(H)\setminus V(P)$.
For  $v_{2i}\in\{v_0, v_2,\ldots, v_{2k}\}$, choose $v\in V(H)\setminus V(P)$. Since $N_H(v)=\binom{\{v_1,v_3,...,v_{2k-1}\}}{r-1}$, there are two distinct  edges $e', e''\in E(H)$ such that $v_{2i-1}, v\in e'$ and $v, v_{2i+1}\in e''$.
So,  by replacing $v_{2i}$ with $v$ and $e_{2i}, e_{2i+1}$ with $e', e''$ in $P$, we get a new Berge path $P'$ of length $2k$. By the symmetry of $v_{2i}$ and $v$, we have $N_H(v_{2i})=\binom{\{v_1,v_3,...,v_{2k-1}\}}{r-1}$. Therefore for all $v\in V(H)\setminus\{v_1,v_3,...,v_{2k-1}\}$, $N_H(v)=\binom{\{v_1,v_3,...,v_{2k-1}\}}{r-1}$. This implies that $S'_r(n,k)\subseteq H\subseteq S_r(n,k)$.

\medskip
\noindent\textbf{Case 2.} For any longest Berge path $P$ in $H$ with $E(P)=\{e_1, \ldots, e_{2k}\}$ and $K(P)=\{v_0, v_1, \ldots, v_{2k}\}$, $\max (\kappa_0(P)-1)<\min \kappa_{2k}(P)$.

Fix a longest Berge path $P$  with $E(P)=\{e_1, \ldots, e_{2k}\}$ and $K(P)=\{v_0,  \ldots, v_{2k}\}$ in $H$.
Since $(\kappa_0-1)\cap \kappa_{2k}=\emptyset$, $(\kappa_0-1)\cup \kappa_{2k}=[0,2k-1]$ and $|\kappa_0|=|\kappa_{2k}|=k$, we have $\kappa_0-1=[0,k-1]$ and $\kappa_{2k}=[k,2k-1]$. By Claim~\ref{CLA: c4}, we have $\kappa_v\cap(\kappa_v-1)=\emptyset$ for any $v\notin K(P)$ and $V(P)=K(P)$.

Since  $H$ is connected and $n>2k+1$, we have $V(H)\setminus V(P)\not=\emptyset$ and $E(H)\setminus E(P)\not=\emptyset$. We claim that for every edge $e\in E^O_v(P)$ with $v\notin V(P)$, if $e\cap V(P)\neq\emptyset$ then $e\cap V(P)=\{v_{k}\}$. Suppose there is $v_i\in e\cap V(P)$ for some $i\neq k$. Then $i\in\kappa_v(P)$. If $i<k$ then $i\in[0,k-1]= \kappa_0-1$, a contradiction to $\kappa_v\cap(\kappa_0-1)=\emptyset$ ((1.2) of Proposition~\ref{PROP: set}). Now assume $i>k$. Then $i-1\in[k,2k-1]=\kappa_{2k}$, a contradiction to $\kappa_v\cap(\kappa_{2k}+1)=\emptyset$ ((1.2) of Proposition~\ref{PROP: set}).
The claim follows.

Since $n>2k+1$, there exists integer $s\ge 3$ such that $(s-1)k+2\le n\le sk+1$.
In the following we will show $H\cong S(sK_{k+1}^{(r)},1)$ with $v_k$ as the center vertex by induction on $s$. For the base case $s=3$, denote $V(H)=V(P)\cup R=\{v_0,v_1,...,v_{2k}\}\cup R$.
So $|R|=n-2k-1\le k$. And for any $v\in R$, by the above claim, we have $N_H(v)\subseteq\binom{(R\setminus\{v\})\cup\{v_k\}}{r-1}$. Since $d_H(v)\ge \binom{k}{r-1}$ and $|(R\setminus\{v\})\cup\{v_k\}|\le k$, we have $|R|=k$ and $N_H(v)=\binom{(R\setminus\{v\})\cup\{v_k\}}{r-1}$. To show $H\cong S(3K_{k+1}^r, 1)$, it is sufficient to show that $H[V(P)]\cong S(2K_{k+1}^r, 1)$.
For $i\in[0,k-1]=\kappa_0-1$, let $e\in E(H)$ with $v_i\in e$, we claim that $e\subseteq \{v_0, \ldots, v_k\}$. If not, suppose there is $v_j\in e$ with $j\in [k+1, 2k]=\kappa_{2k}+1$. If $e\notin E(P)$ then $e\in E_i^O(P)\cap E_j^O(P)$. This is impossible since   by (1.4) of Proposition~\ref{PROP: set}, for $i+1\in\kappa_0$, $j-1\in\kappa_{2k}$ and $i+1\le j-1$, we have $E_i^O(P)\cap E_j^O(P)=\emptyset$. Now suppose $e\in E(P)$. Let $e=e_h$ for some $h\in [1,2k]$.
If $h\in[1,k]=\kappa_0$, by (1.5) of Proposition~\ref{PROP: set}, $v_{s+1}\notin e_h$ for any $s\in\kappa_{2k}$, a contradiction to $v_j\in e$ since $j-1\in\kappa_{2k}$. If $h\in[k+1, 2k]=\kappa_{2k}+1$, then $h-1\in\kappa_{2k}$. By (1.5) of Proposition~\ref{PROP: set}, $v_{s-1}\notin e_h$ for any $s\in\kappa_{0}$, a contradiction to $v_i\in e$ since $i+1\in\kappa_{0}$. With similar discussion, we have for $i\in[k+1, 2k]=\kappa_{2k}+1$ and all edges $e\in E(H)$ with $v_i\in e$, $e\subseteq \{v_k, \ldots, v_{2k}\}$. Therefore, $N_{H}(v_i)\subseteq\binom{\{v_0, \ldots, v_{k}\}\setminus\{v_i\}}{r-1}$ for $i\in[0,k-1]$, and $N_{H}(v_i)\subseteq\binom{\{v_k, \ldots, v_{2k}\}\setminus\{v_i\}}{r-1}$ for $i\in[k+1,2k]$. Since $\delta_1(H)\ge \binom{k}{r-1}$, we have all '$\subseteq$'s are '$=$'s. So $H[V(P)]\cong S(2K_{k+1}^r, 1)$ with $v_k$ as the center vertex.

Now assume the statement holds for $s\ge 3$. For $s+1$, let $H'=H-\{v_0,  \ldots, v_{k-1}\}$. Then $|V(H')|\in[(s-1)k+2, sk+1]$ and $\delta_1(H')\ge d$ since  $d_{H'}(v_k)>\binom{|\{v_{k+1},\ldots,v_{2k}\}|}{r-1}=\binom{k}{r-1}$ and the degrees of the rest vertices remain unchanged in $H'$. Since $H'$ is the subgraph of $H$ and $\ell(H)=2k$, we have $\ell(H')=2k$ and for any longest Berge path $Q$ in $H'$ with $K(Q)=\{u_0, u_1, \ldots, u_{2k}\}$, $\max (\kappa_0(Q)-1)<\min \kappa_{2k}(Q)$ (in fact, it is easy to show that $u_k=v_k\in K(Q)$ by the connectivity of $H'$).
By induction hypothesis, we have $|V(H')|=sk+1$ and $H'\cong S(sK_{k+1}^{(r)},1)$ with $v_k$ as the center,  which implies $H\cong S((s+1)K_{k+1}^{(r)},1)$ with the center $v_k$. We are done.
\end{proof}


Theorem~\ref{THM: Main1} (1) follows from Theorems~\ref{THM: weak} and \ref{THM: 2k} immediately. The following  corollary of Theorem~\ref{THM: 2k} will be used in the proof of Theorem~\ref{THM: Main1} (2).
\begin{cor}\label{COR: kk}
For $k\ge r\ge 2$ and $n>2k+1$, let $H$ be a connected $n$-vertex $r$-graph  with $\delta_1(H)>\binom{k}{r-1}$. If $\ell(H)=2k$ then for any longest Berge path $P$ in $H$, $|\kappa_0(P)|=|\kappa_{2k}(P)|=k$ does not hold.
\end{cor}

Now it is ready to prove Theorem~\ref{THM: Main1} (2). We only need to cope with the cases $k>r=3$ and $k=r\ge 3$. The following is a simple observation.
\begin{obs}\label{OBS: o1}
Let $A$ be a set of integers. If $(A-1)\cup\{a_1, \ldots, a_s\}=A\cup\{b_1, \ldots, b_s\}$ (resp. $(A+1)\cup\{a_1, \ldots, a_s\}=A\cup\{b_1, \ldots, b_s\}$) with $a_1<\ldots<a_s$ and $b_1<\ldots<b_s$, then $A=\cup_{i=1}^s[b_i+1, a_i]$ (resp. $A=\cup_{i=1}^s[a_i, b_i-1]$).
\end{obs}

\subsection{Proof of k$>$r=3.}

\begin{thm}\label{THM: weak2}
Let $k>3$ and $n>4k$. If $H$ is a connected $n$-vertex $3$-graph  with $\delta_1(H)>\binom{k}{2}$, then $\ell(H)\ge 2k+1$.
\end{thm}
\begin{proof} Suppose to the contrary that $\ell(H)\le 2k$.
Let $P$ be a longest Berge path in $H$ with $E(P)=\{e_1,\ldots,e_{t}\}$ and $K(P)=\{v_0, v_1, \ldots, v_{t}\}$. Then $t\le 2k$. Moreover, if $t=2k$ then $|\kappa_0|\not=|\kappa_{2k}|$ by Corollary~\ref{COR: kk}.
By (2.1) of Proposition~\ref{PROP: set},
$|\varepsilon^i_0|+|\varepsilon^i_t|=|(\varepsilon^i_0-1)\cup \varepsilon^i_t|\le|[0,t]|\le 2k+1$. Without loss of generality, assume $|\varepsilon^i_0|\ge |\varepsilon^i_t|$. So $|\varepsilon^i_t|\le k$. By Corollary~\ref{PROP: p1}, we have
\begin{equation}\label{EQN: e4}
\binom{k}{2}<d_H(v_a)\le\binom{|\kappa_a|}{2}+|\varepsilon^i_a|
\end{equation}
for $a\in\{0, t\}$.
This implies that $|\kappa_t|\ge k-1$. By (1.3) of Proposition~\ref{PROP: set},
\begin{equation}\label{EQN: e5}
(\varepsilon^i_0-1)\cap \kappa_t=\varepsilon^i_t\cap \kappa_0=\emptyset \mbox{ and so } |\kappa_a|\le t-|\varepsilon^i_{t-a}| \mbox{ for } a\in\{0,t\}.
\end{equation}
Therefore, we have $|\varepsilon_t^i|\le |\varepsilon_0^i|\le t-k+1\le k+1$ and the equality holds if and only if $t=2k$, $|\kappa_t|=k-1$ and $|\varepsilon_t^i|=k$. Again by (\ref{EQN: e4}) and $|\varepsilon_0^i|\le k+1$, we have $|\kappa_0|\ge k-1$. By (1.2) of Proposition~\ref{PROP: set}, \begin{equation}\label{EQN: e6}
(\kappa_0-1)\cap \kappa_t=\emptyset \mbox{ and } (\kappa_0-1)\cup \kappa_t\subseteq[0, t-1].
\end{equation}

If $|\varepsilon^i_0|= k+1$, then we have $t=2k$, $|\kappa_t|=k-1$ and $|\varepsilon_t^i|=k$.
By (\ref{EQN: e5}), we have $\varepsilon^i_0-1=[0,2k-1]\setminus\kappa_{2k}$. By (2.2) of Proposition~\ref{PROP: set}, $(\varepsilon^i_0-2)\cap \kappa_{2k}=\emptyset$. So $\varepsilon^i_0-2\subseteq[-1,2k-1]\setminus\kappa_{2k}= (\varepsilon^i_0-1)\cup\{-1\}$. By Observation~\ref{OBS: o1}, $\varepsilon^i_0-1=[0, \max(\varepsilon_0^i-1)]$. Since $|\varepsilon^i_0-1|=|\varepsilon_0^i|=k+1$, we have $\varepsilon^i_0=[1, k+1]$. Combining with $(\varepsilon^i_0-1)\cup \varepsilon^i_t=[0,t]$ and $(\varepsilon^i_0-1)\cap \varepsilon^i_t=\emptyset$, we have  $\varepsilon^i_{2k}=[k+1,2k]$. So $v_0, v_{2k}\in e_{k+1}$. But $H$ is a $3$-graph means $e_{k+1}=\{v_0, v_k, v_{k+1}\}=\{v_{2k},v_k,v_{k+1}\}$, this is impossible.

If $|\varepsilon^i_0|=k$, we claim that $|\varepsilon^i_t|=k.$  If not, by (\ref{EQN: e4}) again, we have $|\kappa_t|\ge k$.  By (\ref{EQN: e5}), we get $|\kappa_t|=k$ and $t=2k$. Again by (\ref{EQN: e5}) and (2.2) of Proposition~\ref{PROP: set} and Observation~\ref{OBS: o1}, we have $\varepsilon^i_0=[1,k]$ and $\kappa_{2k}=[k, 2k-1]$. By (\ref{EQN: e6}), we have  $\kappa_0-1\subseteq [0, 2k-1]\setminus \kappa_{2k}=[0, k-1]$, i.e. $\kappa_0\subseteq [1,k]=\varepsilon^i_0$. Since $t=2k$, $|\kappa_0|\not=|\kappa_t|=k$. Hence $|\kappa_0|=k-1$.
By (\ref{EQN: e4}), we have $N_H(v_0)=\binom{K_0}{2}\cup (E_0^I-\{v_0\})$ and $\binom{K_0}{2}\cap (E_0^I-\{v_0\})=\emptyset$, which implies that $\{v_0,v_s,v_{s+1}\}\in E_0^O(P)$ for some $s, s+1\in \kappa_0\subseteq \varepsilon^i_0$, a contradiction to $\{v_0,v_s,v_{s+1}\}=e_{s+1}\in E(P)$.
Therefore, if $|\varepsilon^i_0|=k$ then $|\varepsilon^i_t|=k$. Moreover, the above discussion also implies that the case $|\varepsilon^i_0|=k$, $|\kappa_t|=k$ and $|\kappa_0|=k-1$ does not happen. Since $|\varepsilon^i_t|=k$, (\ref{EQN: e5}) implies $|\kappa_0|\le k$.
If $|\kappa_{t}|=k$, by (\ref{EQN: e5}), we have $t=2k$. So $|\kappa_0|\not=|\kappa_{2k}|=k$. This forces $|\kappa_0|=k-1$. But this case does not happen by the above statement. So we assume $|\kappa_t|=k-1$. Since $|\varepsilon^i_0|=|\varepsilon^i_t|=k$, by the symmetry of $0$ and $t$, the case $|\varepsilon^i_t|=k$, $|\kappa_0|=k$ and $|\kappa_t|=k-1$ does not hold too.  Thus, we only need to consider the case $|\kappa_0|=|\kappa_{2k}|=k-1$.
By (\ref{EQN: e4}), we have $$N_H(v_a)=\binom{K_a}{2}\cup (E_a^I-\{v_a\}) \mbox{ and } \binom{K_a}{2}\cap (E_a^I-\{v_a\})=\emptyset \mbox{ for } a\in\{0,t\}.$$
By (1.3),(2.2) and (2.1) of Proposition~\ref{PROP: set}, we have $\kappa_0, \kappa_0-1,\varepsilon_0^i-1\subseteq[0, t-1]\setminus \varepsilon^i_{t}$. So 
$|(\kappa_0-1)\cap \kappa_0\cap(\varepsilon^i_0-1)|\ge|\kappa_0-1|+|\kappa_0|+|\varepsilon^i_0-1|-2|[0,t-1]\setminus\varepsilon^i_{t}|=k-2\ge 2.$
Since $\kappa_0\subseteq[1,t-1]$, $0, t-1\not\in(\kappa_0-1)\cap \kappa_0\cap(\varepsilon^i_0-1)$. So we can pick $s\in(\kappa_0-1)\cap \kappa_0\cap(\varepsilon^i_0-1)$ with $s\in[1,t-2]$ , i.e. $s,s+1\in \kappa_0$ and $s+1\in \varepsilon^i_0$. Hence $\{v_0,v_{s},v_{s+1}\}\in E^O_0(P)$, but $e_{s+1}=\{v_0, v_s, v_{s+1}\}\in E_0^I(P)$, a contradiction.

Now we assume $|\varepsilon^i_0|\le k-1$, i.e. $|\varepsilon^i_t|\le|\varepsilon^i_0|\le k-1$. By (\ref{EQN: e4}), we have $|\kappa_0|,|\kappa_t|\ge k$. On the other hand, by (\ref{EQN: e6}), $2k\le|\kappa_0|+|\kappa_t|\le|(\kappa_0-1)\cup \kappa_t|\le|[0,t-1]|=t\le 2k$. So we have $t=2k$ and $|\kappa_0|=|\kappa_t|=k$, this is impossible.
\end{proof}

\subsection{Proof of $k=r\ge3$.}
\begin{thm}\label{THM: weak3}
For $r\ge 3$ and $n>2r(r-1)$, let $H$ be a connected $r$-graph on $n$ vertices. If $\delta_1(H)>r$ then $\ell(H)\ge 2r+1$.
\end{thm}
\begin{proof} Suppose to the contrary that $\ell(H)\le 2r$.
Let  $P$ be a longest Berge path in $H$ with $E(P)=\{e_1, \ldots, e_{t}\}$ and $K(P)=\{v_0, v_1, \ldots, v_{t}\}$.
Then $t\le 2r$, and if $t=2r$ then $|\kappa_0|\not=|\kappa_{2r}|$ by Corollary~\ref{COR: kk}.
Let $a_s(P)=d_{H}(v_s)-|E^I_s(P)|$ for $s\in\{0,t\}$.
Then we have the following observations: If $a_s(P)=0$ then $|\kappa_s|=0$ and $|E^I_s|\ge r+1$; if $a_s(P)=1$ then $|\kappa_s|=r-1$ and $|E^I_s|\ge r$; and if $a_s(P)\ge 2$ then $|\kappa_s|\ge r$.  Without loss of generality, assume  $a_0(P)\ge a_t(P)$. Note that the corresponding version of (\ref{EQN: e4})
\begin{equation}\label{EQN: e7}
r<d_H(v_a)\le\binom{|\kappa_a|}{r-1}+|\varepsilon^i_a|,
\end{equation}
for $a\in\{0,t\}$,
(\ref{EQN: e5}) and (\ref{EQN: e6})  still hold.

\medskip
\noindent\textbf{Case 0.} $a_0(P)=a_t(P)=0$.

Then $|E^I_0|+|E^I_t|\ge 2r+2$. But, by (2.1) of Proposition~\ref{PROP: set}, $|E^I_0|+|E^I_t|=|(\varepsilon^i_0-1)\cup \varepsilon^i_t|\le t+1\le 2r+1$, a contradiction.

\begin{claim}\label{CLA: c5}
If $|\kappa_s|=r-1$ then $|\varepsilon_{t-s}^i|\not=t-r+1$ for $s\in\{0,t\}$.
\end{claim}
We give the proof for $s=0$, the case $s=t$ can be proved similarly. Suppose to the contrary that $|\varepsilon_{t}^i|=t-r+1$. Then  $|\kappa_0|+|\varepsilon_t^i|=t$.
By (\ref{EQN: e5}) and (2.2) of Proposition~\ref{PROP: set}, $\varepsilon^i_{t}=[1,t]\setminus\kappa_0$ and $\varepsilon^i_{t}+1\subseteq[1,t+1]\setminus\kappa_0=\varepsilon_t^i\cup\{t+1\}$ i.e. $(\varepsilon^i_{t}+1)\cup\{\min\varepsilon_t^i\}=\varepsilon_t^i\cup\{t+1\}$.
By Observation~\ref{OBS: o1}, $\varepsilon^i_{t}=[\min\varepsilon_t^i, t]$. Since $|\varepsilon_t^i|=t-r+1$, we have $\varepsilon_t^i=[r,t]$ and so $\kappa_{0}=[1,r-1]$.
Thus $\varepsilon^i_0=[0, t]\setminus \varepsilon^i_{t}+1=[1, r]$, $E^O_0(P)=\{e_0\}$, where $e_0=\{v_0\}\cup K_0$.
We claim that for any $e_i$ with $i\in [1,r-1]$, $e_i\subseteq K(P)$. Otherwise, suppose there is $v\notin K(P)$ and $i\in[1,r-1]$ such that $v\in e_i$.
Let $P_i$ be  obtained from $P$ by replacing $e_i$ by $e_0$. Note that $v_{i-1}, v_i\in e_0$. So $P_i$ is a longest Berge path in $H$ with $K(P_i)=K(P)$. Since $i\in \varepsilon_0^i(P)$, $v_0\in e_i$ but $e_i\notin E(P_i)$. Thus $v_0\in K_v(P_i)$. So $P_i$ is extendible by Observation~\ref{PROP: p0}, which is a contradiction.
The claim also implies that for each $i\in[1,r-1]$, there exists a $j_i\in[r,t]$ such that $v_{j_i}\in e_i$.
But this is impossible. Otherwise, $v_0\in e_i$ since $i\in \varepsilon_0^i$.
Since $j_i\in\varepsilon_t^i$, we have  $C$  with $E(C)=\{e_i, e_{j_i+1},\ldots, e_{t}, e_{j_i},\ldots, e_{i+1}, e_0, e_{i-1},\ldots, e_1\}$ and $K(C)=\{v_0, v_{j_i},\ldots,v_t, v_{j_i-1}, \ldots, v_1\}$ is a Berge cycle of length $t+1$, which is a contradiction to $n>2r(r-1)$ (by Lemma~\ref{LEM: cycle1}, we have $n=c(H)=t+1\le 2r+1$).

\medskip
\noindent\textbf{Case 1.} $a_0(P)=1$ and $a_t(P)=0$.

Then $|\kappa_0|=r-1$, $|E^I_0|\ge r$ and $|\kappa_{t}|=0$, $|E^I_t|\ge r+1$. By $2r+1\le|\varepsilon^i_0|+|\varepsilon^i_t|\le t+1\le 2r+1$, we have  $|\varepsilon^i_0|=r$, $|\varepsilon^i_{t}|=r+1$, and $t=2r$. By Claim~\ref{CLA: c5}, this is impossible.

\medskip
\noindent\textbf{Case 2.} $a_0(P)=1$ and  $a_t(P)=1$.

Then $|\varepsilon^i_s|\ge r$, $|\kappa_s|=r-1$ for $s\in\{0,t\}$, and $t\ge 2r-1$.
If $t=2r-1$
then $|\varepsilon^i_t|\ge r=t-r+1$. By Claim~\ref{CLA: c5}, we get a contradiction.


Now suppose $t=2r$. If $|\varepsilon^i_{s}|\ge r+1(=2r-r+1)$ for some $s\in\{0,t\}$, again by Claim~\ref{CLA: c5}, we have a contradiction.
Thus $|\varepsilon_{0}^i|=|\varepsilon_{2r}^i|=r$.
By (2.1) of Proposition~\ref{PROP: set}, there exists an integer $z\in[1, 2r-1]$ such that $[0,2r]=(\varepsilon^i_0-1)\cup \varepsilon^i_{2r}\cup\{z\}$.
{Without loss of generality, we may assume $\max(\varepsilon_0^i-1)<z$ or $\min\varepsilon_{2r}^i<z$. Otherwise, we have $\max(\varepsilon_0^i-1)>z$ and  $\min\varepsilon_{2r}^i>z$. Then we reverse the order of $P$, i.e. relabel the vertices $v_i$ by $v_{2r-i}$ for $0\le i\le 2r$ and edges $e_j$ by $e_{2r+1-j}$ for $1\le j\le 2r$, and denote the reversed path by $P'$. Hence $\varepsilon^i_0(P')=2r+1-\varepsilon^i_{2r}(P)$, $\varepsilon^i_{2r}(P')=2r+1-\varepsilon^i_{0}(P)$, and $[0,2r]=(\varepsilon^i_0(P')-1)\cup \varepsilon^i_{2r}(P')\cup\{z'\}$, where $z'=2r-z$. So, $$\max(\varepsilon_0^i(P')-1)=\max(2r-\varepsilon_{2r}^i(P))=2r-\min\varepsilon_{2r}^i(P)<2r-z=z'$$ and $$\min\varepsilon_{2r}^i(P')=\min(2r+1-\varepsilon_0^i(P))=2r-\max(\varepsilon_0^i(P))<2r-z=z'.$$
Therefore, we can reverse the order of $P$ instead if any.
}

\begin{claim}
$\min\varepsilon_{2r}^i<\max(\varepsilon_0^i-1)$.
\end{claim}
If not, then $\varepsilon_{2r}^i=[r+1, 2r]$ ($z\le r$) or $[r, 2r]\setminus\{z\}$ (if $z>r$).

If $\varepsilon_{2r}^i=[r+1, 2r]$, then $\kappa_0=[1,r]\setminus\{x\}$ for some $1\le x\le r$. Let $e_0=\{v_0\}\cup K_0$. Then $e_0\notin E(P)$.
Since $\min\varepsilon_{2r}^i=r+1>z$, we have $\max(\varepsilon_0^i-1)<z$, which implies $z=r$, $\varepsilon_0^i-1=[0,r-1]$ (or equivalently, $\varepsilon_0^i=[1,r]$).
We claim that for any $e_i$ with $i\in[1,r-1]$ or $[1,r]\setminus\{x+1\}$ (if $x<r$), $e_i\subseteq K(P)$. Otherwise, suppose there is $v\notin K(P)$ and $i\in[1,r-1]$ or $[1,r]\setminus\{x+1\}$ (if $x<r$) such that $v\in e_i$.
Let $P_i$ be  obtained from $P$ by replacing $e_i$ by $e_0$ if $i\not=x$ and let $P_i$ be  obtained from $P$ by replacing $e_{i+1}$ by $e_0$ and reversing the order of $e_1, \ldots, e_{i}$ and their corresponding key vertices if $i=x$. Note that $v_{i-1}, v_i\in e_0$ if $i\not=x$ and $v_0,v_{i+1}\in e_0$ if $i=x$ (since $x<r$). So $P_i$ is a longest Berge path in $H$ with $K(P_i)=K(P)$ (remark: the first element of $K(P_i)$ is $v_i$ for $i=x$). Since $i\in \varepsilon_0^i(P)$, $v_0\in e_i$ but $e_i\notin E(P_i)$. Thus $v_0\in K_v(P_i)$ if $i\not=x$ or $v_i\in K_v(P_i)$ if $i=x$. So $P_i$ is extendible by Observation~\ref{PROP: p0}, which is a contradiction.
Next we claim that for each $i\in[1,r-1]$ or $[1,r]\setminus\{x+1\}$, $e_i\setminus\{v_0\}\subseteq\{v_1,\ldots, v_r\}$.  If not, suppose there exists such an $i$ and a $j_i\in[r+1,2r]$ such that $v_{j_i}\in e_i$.
Since $i\in \varepsilon_0^i$, we have $v_0\in e_i$.
Since $j_i\in\varepsilon_{2r}^i$, we have  $C$  with $E(C)=\{e_i, e_{j_i+1},\ldots, e_{2r}, e_{j_i},\ldots, e_{i+1}, e_0, e_{i-1},\ldots, e_1\}$ and $K(C)=\{v_0, v_{j_i},\ldots,v_{2r}, v_{j_i-1}, \ldots, v_1\}$ is a Berge cycle of length $2r+1$, which is a contradiction to $n>2r(r-1)$ (by Lemma~\ref{LEM: cycle1}, we have $n=c(H)\le 2r+1$).
Let $A=[1,r-1]$ or $[1,r]\setminus\{x+1\}$. Then $|A\cup\{0\}|=r$ and for each $i\in A\cup\{0\}$, we have $e_i\setminus\{v_0\}\subseteq\{v_1, \ldots, v_r\}$.
Since $\{v_1, \ldots, v_r\}$ has exactly $r$ subsets of size $r-1$, we have $\{e_i : i\in A\cup\{0\}\}={\{v_1,\ldots, v_r\}\choose {r-1}}$. So replacing $\{e_1,\ldots, e_r\}$ by $\{e_i : i\in A\cup\{0\}\}$ with suitable order in $P$, we get another longest Berge path $Q$ with $K(Q)=K(P)$ in $H$. Clearly, $[r+1,2r]\subseteq\varepsilon_{2r}^i(Q)$. So, $\kappa_0(Q)\subseteq [1,r]$. Denote $\{y\}=[1,r]\setminus A$. Then $e_y\notin E(Q)$ and $v_0\in e_y$. This forces that $e_y=K_0(Q)\cup\{v_0\}\subseteq\{v_0, v_1,\ldots, v_r\}$. But this is a contradiction to $\{e_i : i\in A\cup\{0\}\}={\{v_1,\ldots, v_r\}\choose {r-1}}$.



So $\varepsilon_{2r}^i=[r, 2r]\setminus\{z\}$. Thus $\varepsilon_0^i-1=[0,r-1]$. This forces that $\kappa_{2r}\subseteq[r,2r-1]$, a contradiction to $\max(\varepsilon_0^i-1)>\min\kappa_{2r}$. The claim follows.

So $\varepsilon_{2r}^i$ (resp. $\varepsilon_0^i-1$) consists of at least two consecutive intervals.
Hence $|\varepsilon_{2r}^i\cup(\varepsilon_{2r}^i+1)|\ge r+2$ (resp. $|(\varepsilon_0^i-1)\cup(\varepsilon_0^i-2)|\ge r+2$). By (\ref{EQN: e5}) and (2.2) of Proposition~\ref{PROP: set}, $|\varepsilon^i_{2r}\cup(\varepsilon^i_{2r}+1)|=r+2$  and $\kappa_0\cup\varepsilon^i_{2r}\cup(\varepsilon^i_{2r}+1)=[1,2r+1]$. Denote $\min\varepsilon_{2r}^i=a$.
By Observation~\ref{OBS: o1}, $\varepsilon_{2r}^i$  consists of exactly two consecutive intervals. So we may assume $\varepsilon_{2r}^i=[a, a+p]\cup[r+p+2, 2r]$ for some integer $p\ge 0$. Thus $\kappa_0=[1,a-1]\cup[a+p+2, r+p+1]$.
 Let $e'=K_0\cup\{v_0\}$. Then $e'\notin E(P)$. Since $|e'|=|\kappa_0\cup\{0\}|=|[0,a-1]\cup[a+p+2,r+p+1]|=r\ge 3$, there exists $i\in\kappa_0$ such that $v_{i-1}, v_i\in e'$.
 Now let $P'$ be obtained from $P$ by replacing $e_i$ with $e'$. Then $P'$ is another longest Berge path with $K(P')=K(P)$ and the same $\varepsilon^i_{2r}$ with $P$. So $\kappa_{0}(P')=\kappa_0(P)$ too. This implies that $e_i=e'$, a contradiction.

\medskip

\noindent\textbf{Case 3.} $a_0\ge 2$.

Then $|\kappa_0|\ge r$.

If $a_t=0$ then $|\varepsilon^i_t|\ge r+1$. This is impossible, since, by (\ref{EQN: e5}), $|\kappa_0|+|\varepsilon^i_t|=|\kappa_0\cup \varepsilon^i_t|\le t\le 2r$.

If $a_t=1$ then  {$|\kappa_t|=r-1$ and $|\varepsilon^i_t|\ge r$. By (\ref{EQN: e5}), we have $\varepsilon_t^i\cap\kappa_0=\emptyset$. So  $|\kappa_0|=|\varepsilon^i_t|=r$, $\varepsilon_t^i\cup\{2r+1\}=[1,2r+1]\backslash\kappa_0$ and $t=2r$. By (2.2) of Proposition~\ref{PROP: set}, $(\varepsilon_t^i+1)\cup\{\min\varepsilon_t^i\}=[1,2r+1]\backslash\kappa_0$. By Observation~\ref{OBS: o1}, $\varepsilon^i_t=[r+1,2r]$ and so $\kappa_0=[1,r]$. By (\ref{EQN: e6}),  $\kappa_t=[r,2r-1]\backslash\{i_0\}$ for some $i_0\in[r,2r-1]$. So $E_t^O=\{e_0\}$, where $e_0=\{v_t\}\cup K_t(P)=\{v_r, \ldots, v_{2r}\}\backslash\{v_{i_0}\}$.

Now we claim that for any $i\in[r+1,2r]$, $e_i\subset\{v_r,...,v_{2r}\}$. Else, suppose $j\in[r+1,2r]$ is a counterexample, i.e, there exists $v\notin\{v_r,...,v_{2r}\}$ with $v\in e_j$. If $j\neq i_0+1$ and $v\notin K(P)$, then we get a longer Berge path $P'$ of length $2r+1$ with $E(P')=\{e_1,...,e_{j-1},e_0,e_{2r},...,e_{j+1},e_j\}$ and $K(P')=\{v_0,...,v_{j-1},v_{2r},..,v_{j},v\}$, a contradiction. If $j\neq i_0+1$ but $v\in K(P)$, assume $v=v_s$ with $s\in[0,r-1]$. Since $s+1\in\kappa_0=[1,r]$, there is an edge $e\in E(H)\setminus E(P)$ containing $v_0$ and $v_{s+1}$. Hence we get a Berge cycle $C$ of length $2r+1$ with $K(C)=\{v_0,v_{s+1},v_{s+2},...,v_{j-1},v_{2r},...,v_j,v_s,...,v_1\}$ and $E(C)=\{e, e_{s+2}, \ldots, e_{j-1}, e_0, e_{2r}, \ldots, e_j, e_s,\ldots, e_1\}$, which is a contradiction by Lemma~\ref{LEM: cycle1}.
Let $$A=\left\{e\in E(H): e=\{v_{2r}\}\cup X, \mbox{ where } X\in\binom{\{v_r,...,v_{2r-1}\}}{r-1}\right\}.$$ 
Note that $\varepsilon_{2r}^i=[r+1, 2r]$. So we have, for all $i\in\{0\}\cup[r+1,2r]\backslash\{i_0+1\}$, $e_i\backslash\{v_{2r}\}\in\binom{\{v_r,...,v_{2r-1}\}}{r-1}$.
Since $|\{0\}\cup[r+1,2r]\backslash\{i_0+1\}|=r=|\binom{\{v_r,...,v_{2r-1}\}}{r-1}|$, we get
$\{e_0\}\cup E_{2r}^I\backslash\{e_{i_0+1}\}=A$.
Note that $A$ induces an almost complete $r$-graph on $\{v_r, \ldots, v_{2r}\}$, i.e. $H[A]\cong K_{r+1}^{(r)}-e$, where $e=\{v_r,...,v_{2r-1}\}$. By the symmetry of $v_i$'s for $i\in[r,2r-1]$, it is easy to check that for any order of the vertices in $\{v_r, \ldots, v_{2r}\}$, there exists a Berge path of length $r$ corresponding to it.
Now suppose $j=i_0+1\in[r+1,2r]$. If $v\notin K(P)$, we have a Berge path $P'$ of length $2r+1$ with $K(P')=\{v_0,\ldots,v_{j-1},v_{2r},..,v_{j},v\}$, a contradiction.
Similarly, if $v=v_s\in K(P)$ for some $s\in[0,r-1]$, we get a Berge cycle $C$ of length $2r+1$ with $K(C)=\{v_0,v_{s+1},...,v_{j-1},v_{2r},...,v_j,v_s,...,v_1\}$, also a contradiction by Lemma~\ref{LEM: cycle1}.
Therefore, we have the claim.
But the claim implies $\{e_0\}\cup E_{2r}^I\subseteq A$, this is impossible since $|\{e_0\}\cup E_{2r}^I|=r+1>r=|A|$, a contradiction.}

Now assume $a_t\ge 2$. Then $|\kappa_t|\ge r$. By (\ref{EQN: e6}), $|\kappa_t|+|\kappa_0|\le t\le 2r$. This forces that $t=2r$ and $|\kappa_t|=|\kappa_0|=r$. But this is impossible by Corollary~\ref{COR: kk}.

The proof is completed.
\end{proof}

\section{An application}
To prove Theorem~\ref{THM: Main3}, we first give a lemma.
\begin{lem}\label{LEM: Cor-Lem}
Suppose $r\ge 2$ and $H$ is a connected $n$-vertex $r$-graph with $\ell(H)=t$.  If $\delta_1(H)>\binom{\lfloor\frac{t}{2}\rfloor}{r-1}+\lceil\frac{t}{2}\rceil$, then $n=t+1$ and $H$ contains a Berge Hamiltonian cycle.
\end{lem}
\begin{proof}
Suppose to the contrary that $n>t+1$ or $n=t+1$ but $H$ does not contain a Berge Hamiltonian cycle.
Let $P$ be a longest Berge path with $E(P)=\{e_1,\ldots, e_t\}$ and $K(P)=\{v_0, v_1, \ldots, v_t\}$.
By the Remark of Proposition~\ref{PROP: set}, (1.1), (1.2) and (1.3) hold.
Now by (1.1) and (1.2) of Proposition~\ref{PROP: set}, we have $(\kappa_0-1)\cap\kappa_t=\emptyset$ and $(\kappa_0-1)\cup\kappa_t\subseteq[0,t-1]$, which implies that $|K_0|+|K_t|\le t$. Similarly by (1.3) of Proposition~\ref{PROP: set}, we have $|E_0^I|\le t-|K_t|$.
Without loss of generality, suppose $|K_0|\le |K_t|$. Then $|K_0|\le\lfloor\frac{t}{2}\rfloor$. By (b) of Corollary~\ref{PROP: p1},
$$d_H(v_0)\le\binom{|K_0|}{r-1}+|E_0^I|\le\binom{|K_0|}{r-1}+t-|K_t|\le\binom{|K_0|}{r-1}+t-|K_0|\mbox{.}$$
By the convexity of the function $f(x)=\binom{x}{r-1}+t-x$ for $x\in\left[0,\lfloor\frac{t}{2}\rfloor\right]$, we have
$$d_H(v_0)\le \max\left\{f(0),f\left(\left\lfloor\frac{t}{2}\right\rfloor\right)\right\}=\binom{\left\lfloor\frac{t}{2}\right\rfloor}{r-1}+\left\lceil\frac{t}{2}\right\rceil<\delta_1(H)\mbox{,}$$
a contradiction.
\end{proof}

Now, it is ready to give the proof of Thoerem~\ref{THM: Main3}.

\begin{proof}[Proof of Theorem~\ref{THM: Main3}]
First, we claim that $H$ is connected. If not, then there are two distinct vertices $u_1,u_2\in V(H)$ such that there is no Berge path connecting them.
For any $v\in V(H)$, let $$A_v:=\{u\in V(H)\backslash\{v\}: d_2(\{u,v\})\ge 1\}.$$
 Then $d_H(v)\le\binom{|A_v|}{r-1}$ for any $v\in V(H)$.
Clearly, $u_1\notin A_{u_2}$ and $u_2\notin A_{u_1}$ and $A_{u_1}\cap A_{u_2}=\emptyset$.
Therefore, $|A_{u_1}|+|A_{u_2}|\le|V(H)\setminus\{u_1,u_2\}|=n-2$. Without loss of generality, assume $|A_{u_1}|\le|A_{u_2}|$. Then $|A_{u_1}|\le\lfloor\frac{n-2}{2}\rfloor$. So $d_H(u_1)\le\binom{\lfloor\frac{n-2}{2}\rfloor}{r-1}<\delta_1(H)$, a contradiction.

If $n\ge 2r+4$ is even, then there exists an integer $k>r$ so that $n=2k+2>2k+1$. Since $\delta_1(H)>{k\choose r-1}$, by Theorem~\ref{THM: Main1}, we get $\ell(H)\ge 2k+1$. Since a Berge path $P$ of length $2k+1$ contains $2k+2=n$ key vertices, $P$ is a longest Berge path, i.e. $\ell(H)=2k+1$. Since $\delta_1(H)>{k\choose r-1}+k+1$, by Lemma~\ref{LEM: Cor-Lem}, $H$ contains a Berge Hamiltonian cycle.

If $n\ge 2r+5$ is odd, then there exists an integer $k>r$ such that $n=2k+3>2k+1$. Similarly, since $\delta_1(H)>{k+1\choose r-1}+k+1>{k\choose r-1}+k+1$ and by Theorem~\ref{THM: Main1}, we have $\ell(H)\ge 2k+1$. If $\ell(H)=2k+1$, by Lemma~\ref{LEM: Cor-Lem}, we have $n=2k+1+1=n-1$, a contradiction. Thus $\ell(H)=2k+2=n-1$. By Lemma~\ref{LEM: Cor-Lem}, $H$ contains a Berge Hamiltonian cycle.
\end{proof}

\section{Remarks}
In this paper, we give the minimum degree threshold for $r$-uniform hypergraphs on $n$ vertices containing no Berge path of length $2k+1$. Furthermore, for $k>r\ge 4$ and $n>2k+1$, we characterize the extremal graphs.  However, we know nothing about $k<r$, we leave this as a problem.

Recently, F\"uredi, Kostochka, and Luo~\cite{FKL20} gave minimum degree thresholds for non-uniform hypergraphs containing no long Berge paths and cycles. But as we have known there is no Dirac-type minimum degree condition for uniform hypergraphs containing no long cycles so far. It is also an interesting problem to determine the  minimum degree condition for uniform hypergraphs containing no long cycles.

\end{document}